\documentclass[letterpaper, 10 pt, conference]{ieeeconf}  

\IEEEoverridecommandlockouts                            
\makeatletter
\let\NAT@parse\undefined
\makeatother
\usepackage{lipsum}
\usepackage{amsfonts}
\usepackage{graphicx}
\usepackage{epstopdf}
\usepackage{standalone}
\usepackage{algcompatible}

\ifpdf
\DeclareGraphicsExtensions{.eps,.pdf,.png,.jpg}
\else
\DeclareGraphicsExtensions{.eps}
\fi

\usepackage{amsopn}

\usepackage{epsfig} % for postscript graphics files
\usepackage{amsmath} % assumes amsmath package installed
\usepackage{amssymb}  % assumes amsmath package installed
\usepackage{bm}
\DeclareMathAlphabet{\mathcal}{OMS}{cmsy}{m}{n}
\usepackage[english]{babel}
\usepackage{url}
\usepackage{algorithm}
\usepackage{color}
\usepackage{pgfplots}
\usepackage{siunitx}
\usepackage{tikz}
\usepackage{tkz-euclide}
\usepackage{chngcntr}
\usepackage{verbatim}

\definecolor{ao(english)}{rgb}{0.0, 0.5, 0.0}
\usepackage[colorlinks, citecolor = {ao(english)}, linkcolor = {ao(english)}]{hyperref} 
\usepackage{cleveref}
\usepackage{aliascnt}
\usetikzlibrary{calc}
\usepackage{siunitx}
\usepackage{subcaption}
\usepackage{tikz}
\usepackage{pgfplots}
\usetikzlibrary{arrows,shapes,trees,calc,positioning,patterns,decorations.pathmorphing,decorations.markings}
\usetikzlibrary{matrix}
\usepgfplotslibrary{groupplots}
\pgfplotsset{compat=newest}
\usepackage{amsthm}

\crefname{figure}{Fig.}{Fig.}

\newtheorem{thm}{Theorem}
\crefname{thm}{Theorem}{Theorems}

\newtheorem{prop}{Proposition}
\crefname{prop}{Proposition}{Propositions}
\newtheorem{lem}{Lemma}
\crefname{lem}{Lemma}{Lemmas}
\newtheorem{cor}{Corollary}
\crefname{cor}{Corollary}{Corollaries}
\theoremstyle{remark}
\newtheorem{rem}{Remark}
\crefname{rem}{Remark}{Remarks}

\theoremstyle{remark}
\newtheorem{example}{Example}
\crefname{rem}{Example}{Examples}

\crefname{ass}{Assumption}{Assumption}
\usepackage{dsfont}
\let\mathbb=\mathds
\usepackage[numbers,sort&compress]{natbib}

\crefname{conj}{Conjecture}{Conjectures}

\theoremstyle{definition}
\newtheorem{defn}{Definition}
\crefname{defn}{Definition}{Definitions}
\crefname{prob}{Problem}{Problems}
\crefname{algorithm}{Algorithm}{Algorithms}

\newcommand{\Rmn}{\mathbb{R}^{n \times m}}

\newcommand{\Rnn}{\mathbb{R}^{n \times n}}

\newcommand{\transp}{\mathsf{T}}

\colorlet{FigColor1}{blue}
\colorlet{FigColor2}{red}
\colorlet{FigColor3}{ao(english)}
\colorlet{FigColor4}{orange}
\pgfplotsset{every axis plot/.append style={line width=1.5pt}}

\crefformat{equation}{\textup{#2(#1)#3}}
\crefrangeformat{equation}{\textup{#3(#1)#4--#5(#2)#6}}
\crefmultiformat{equation}{\textup{#2(#1)#3}}{ and \textup{#2(#1)#3}}
{, \textup{#2(#1)#3}}{, and \textup{#2(#1)#3}}
\crefrangemultiformat{equation}{\textup{#3(#1)#4--#5(#2)#6}}%
{ and \textup{#3(#1)#4--#5(#2)#6}}{, \textup{#3(#1)#4--#5(#2)#6}}{, and \textup{#3(#1)#4--#5(#2)#6}}

\Crefformat{equation}{#2Equation~\textup{(#1)}#3}
\Crefrangeformat{equation}{Equations~\textup{#3(#1)#4--#5(#2)#6}}
\Crefmultiformat{equation}{Equations~\textup{#2(#1)#3}}{ and \textup{#2(#1)#3}}
{, \textup{#2(#1)#3}}{, and \textup{#2(#1)#3}}
\Crefrangemultiformat{equation}{Equations~\textup{#3(#1)#4--#5(#2)#6}}%
{ and \textup{#3(#1)#4--#5(#2)#6}}{, \textup{#3(#1)#4--#5(#2)#6}}{, and \textup{#3(#1)#4--#5(#2)#6}}

\crefdefaultlabelformat{#2\textup{#1}#3}

\renewcommand{\baselinestretch}{1}
\title{\LARGE Strongly unimodal systems
}
\author{Christian Grussler and Rodolphe Sepulchre% <-this % stops a space
	\thanks{The authors are with the Control Group at the Department of Engineering, University of Cambridge, Trumpington Street, Cambridge CB2 1PZ, United Kingdom.
				{\tt\small \{christian.grussler, r.sepulchre\}@eng.cam.ac.uk}}%
\thanks{The research leading to these results has received funding from the European Research Council under the
Advanced ERC Grant Agreement Switchlet n.670645 and from DGAPA-UNAM under the grant PAPIIT RA105518.
}% <-this % stops a space
	}

\begin{document}

	\maketitle
	\thispagestyle{empty}
	\pagestyle{empty}

	%%%%%%%%%%%%%%%%%%%%%%%%%%%%%%%%%%%%%%%%%%%%%%%%%%%%%%%%%%%%%%%%%%%%%%%%%%%%%%%%
	\begin{abstract}
%	This work introduces a new class of externally positive system, which we call log-concave system. The appeal of this class is the property of inheriting unimodality from input to output -- a systems property that we believe is present in many subsystems and thus gives rise to a new system analytical input-output framework. %In particular, it can be seen that this is one of the fundamental concepts which convolutional neural networks (CNNs) are build on. Thus a system theoretical understanding of this system class gives rise to system theoretical treatment of CNNs.   
%	In order to make this property tractable, this work starts with the realization theory for linear systems by providing a state-space characterization based on external positivity. Thus allowing us to test for log-concavity in many linear and non-linear systems as well as to understand its limiting factors.
We investigate the property for an input-output system to map unimodal inputs to unimodal outputs. As a first step, we analyse this property for linear time-invariant (LTI) systems, static nonlinearities, and interconnections of those. In particular, we show how unimodality is closely related to the concepts of positivity, monotonicity, and total positivity.  
	\end{abstract}
	\vspace*{.3 cm}
	%%%%%%%%%%%%%%%%%%%%%%%%%%%%%%%%%%%%%%%%%%%%%%%%%%%%%%%%%%%%%%%%%%%%%%%%%%%%%%%%

	\section{Introduction}

System analysis via the concept of {\it positivity} has gained considerable popularity in the recent years \cite{farina2011positive,rantzer2015scalable,son1996robust,tanaka2011bounded,sootla2012scalable}. 
From a modelling viewpoint, the value of devoting a special treatment to dynamical models that manipulate positive variables (states, inputs, or outputs) was recognised early by Luenberger \cite{luenberger1979introduction}, as this situation frequently arises in  networks, economics, biology, transport, etc.  From an analysis viewpoint, the increasing use of convexity analysis in system theory led a number of authors to revisit the  classical linear-quadratic theory of linear-time invariant (LTI) systems in the presence of positive constraints. Positivity was shown  to be a source of numerical tractability and simplicity even in the standard context of  Lyapunov analysis \cite{rantzer2015scalable}, optimal control design \cite{tanaka2011bounded,ebihara2012optimal}, or   system gain computation \cite{farina2011positive}.  It is surprising that such properties have only started to gain widespread interest.  Positivity concepts have also proven  useful beyond LTI systems. Positivity is  central to consensus and distributed system analysis, which involves linear time-varying models \cite{Sepulchre2010}. It is also central to the theory of monotone systems \cite{Hirsch2006,Angeli2003} and to the recent development of differential positivity analysis \cite{Forni2016,Mostajeran2018a}.

The present paper focuses on the input-output (or external) concept of  positivity: a system is called (externally) positive if it maps positive inputs to positive outputs. For linear systems, this property is equivalent to (external) monotonicity: input signals with a time-derivative that has no sign variation  are mapped to output signals with the same property. In a similar spirit, we aim at characterizing systems that map unimodal inputs to unimodal outputs : input signals with a time-derivative that has at most {\it one} sign variation are mapped to output signals with the same property. We call such systems {\it strongly unimodal}  because a well-known result in probability theory: strongly unimodal densities are precisely those that map unimodal densities to unimodal densities by convolution \cite{ibragimov1956composition}.
 	
While strong unimodality has been extensively studied in statistics and interpolation theory, it does not seem to have received much attention in system theory. Our motivation is that it is nevertheless  a natural property to expect in the context of mean-field models. Classical examples include amplifier modelling  in electronics, conductance modelling in neurophysiology, or reaction rate modelling in biochemistry.  In first approximation, it is natural in all those examples to expect that a unimodal input is mapped to a unimodal output (see e.g. \cite{Sepulchre2018} for details).  This motivation makes direct contact with the questions that have motivated the development of total positivity theory in interpolation theory and in statistics  \cite{karlin1968total,dharmadhikari1988unimodality,Schoenberg1988polyI}.  Starting with the early work of Schoenberg   \cite{Schoenberg1930vari}, the entire theory has been motivated by a characterization of maps with {\it variation-diminishing} properties. We believe that such properties could play an important role as well in system analysis of models grounded in mean-field principles.

As a first step towards a more general theory, the main focus of the present paper is the simple class of  LTI systems. Our main result is to show that strong unimodality of a LTI state-space model can be studied as the external positivity of a {\it compound} LTI state-space model. This methodology is central to the theory of total positivity: properties of unimodal systems   are studied via the properties of a compound {\it externally positive} system. As an elementary application, we single out the main difference between positive and unimodal LTI systems: externally positive systems have one dominant real pole, whereas unimodal systems have two dominant real poles. Using properties of log-concave functions, we derive a number of properties for the interconnections of LTI unimodal systems and monotone functions. While elementary, those preliminary results suggest a strong potential of total positivity theory in system analysis.

The remainder of the paper is organized as follows. After some preliminaries in \Cref{sec:prelim}, the theory of external positive systems is briefly reviewed in \Cref{sec:ex_pos}. Then \Cref{sec:strong_uni} introduces the analog concept of strong unimodality. Our preliminary  results are presented in \Cref{sec:state_space}. \Cref{sec:inter} introduces elementary interconnection properties of unimodal systems. The papers ends with concluding remarks in \Cref{sec:conc}. Proofs are given in the appendix.

\tikzstyle{int}=[draw,minimum width=1cm, minimum height=1cm, very thick, align = center]
\begin{figure}
	\hspace*{0.02 cm}
\begin{center}
%	\begin{tikzpicture}[>=latex']
%	\node (a) [int] {externally positive \\ LTI system}; 
%	\node (b) [left of=a,node distance=4cm, coordinate]{};
%	\node [coordinate] (end) [right of=a, node distance=4cm]{};
%	
%	% Input
%	\path[->,thick] (b) edge node[below]{input} node(u_inc)[above]{} (a);
%	\begin{axis}[ticks = none,width = 3 cm,at=(u_inc), anchor={south}]
%	\addplot[line width = 1 pt, color = FigColor1] file{u_inc.txt};
%	\end{axis}
%	
%	% Output
%	\path[->,thick] (a) edge node[below]{output} node(y_inc)[above,midway]{} (end);
%	\begin{axis}[ticks = none,width = 3 cm,at=(y_inc), anchor={south}]
%	\addplot[line width = 1 pt, color = FigColor2] file{y_inc.txt};
%	\end{axis}
%
%\node (a1) [int, below of = a, anchor = north, yshift = - .5 cm] {strongly unimodal \\ LTI system}; 
%\node (b1) [left of=a1,node distance=4cm, coordinate]{};
%\node [coordinate] (end1) [right of=a1, node distance=4cm]{};
%
%% Input
%\path[->,thick] (b1) edge node[below]{input} node(u_inc)[above]{} (a1);
%\begin{axis}[ticks = none,width = 3 cm,at=(u_inc), anchor={south}]
%\addplot[line width = 1 pt, color = FigColor1] file{u_uni.txt};
%\end{axis}
%
%% Output
%\path[->,thick] (a1) edge node[below]{output} node(y_inc)[above,midway]{} (end1);
%\begin{axis}[ticks = none,width = 3 cm,at=(y_inc), anchor={south}]
%\addplot[line width = 1 pt, color = FigColor2] file{y_uni.txt};
%\end{axis}
%	
%	\end{tikzpicture}
%	
%	DECIDE ON WHICH TO PICK: - THE UPPER ONE LOOKS MORE LIKE AN IDENTITY MAPPING, - THE LOWER SHOWS MORE OF THE SMOOTHING
		\begin{tikzpicture}[>=latex']
	\node (a) [int] {externally positive \\ LTI system}; 
	\node (b) [left of=a,node distance=4cm, coordinate]{};
	\node [coordinate] (end) [right of=a, node distance=4cm]{};
	
	% Input
	\path[->,thick] (b) edge node[below]{input} node(u_inc)[above]{} (a);
	\begin{axis}[ticks = none,width = 3 cm,at=(u_inc), anchor={south}]
	\addplot[line width = 1 pt, color = FigColor1] file{u_inc1.txt};
	\end{axis}
	
	% Output
	\path[->,thick] (a) edge node[below]{output} node(y_inc)[above,midway]{} (end);
	\begin{axis}[ticks = none,width = 3 cm,at=(y_inc), anchor={south}]
	\addplot[line width = 1 pt, color = FigColor2] file{y_inc1.txt};
	\end{axis}
	
	\node (a1) [int, below of = a, anchor = north, yshift = - .5 cm] {strongly unimodal \\ LTI system}; 
	\node (b1) [left of=a1,node distance=4cm, coordinate]{};
	\node [coordinate] (end1) [right of=a1, node distance=4cm]{};
	
	% Input
	\path[->,thick] (b1) edge node[below]{input} node(u_inc)[above]{} (a1);
	\begin{axis}[ticks = none,width = 3 cm,at=(u_inc), anchor={south}]
	\addplot[line width = 1 pt, color = FigColor1] file{u_uni1.txt};
	\end{axis}
	
	% Output
	\path[->,thick] (a1) edge node[below]{output} node(y_inc)[above,midway]{} (end1);
	\begin{axis}[ticks = none,width = 3 cm,at=(y_inc), anchor={south}]
	\addplot[line width = 1 pt, color = FigColor2] file{y_uni1.txt};
	\end{axis}
	
	\end{tikzpicture}
\end{center}
\caption{A positive LTI system maps monotone inputs to monotone outputs. A strongly unimodal LTI system maps unimodal inputs to unimodal outputs.}
	\label{fig:UNIMOD}
\end{figure}
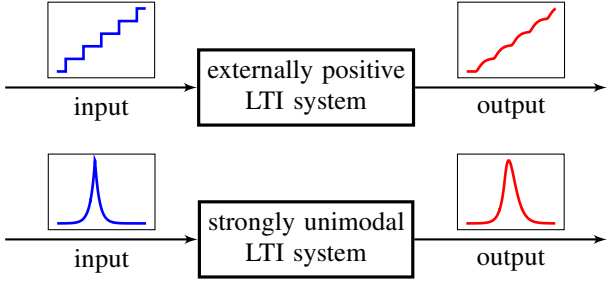

%	\subsection{Outline}
%	This paper is organized as follows. We first introduce some notations, as well as, some background on externally positive systems. Subsequently, we derive our main results on the verification and analysis of strong unimodality. 
	\section{Preliminaries}
	\label{sec:prelim}
	\subsection{Notations}
	For real valued matrices $X = (x_{ij}) \in \Rmn$, including vectors $x = (x_i) \in \mathbb{R}^n$, we use the following notation. $X$ is called \emph{nonnegative}, $X \in \Rmn_{\geq 0}$, if all its entries $x_{ij}$ are nonnegative. If $X \in \Rnn$, then $\sigma(X) = \{\lambda_1(X),\dots,\lambda_n(X)\}$ denotes its \emph{spectrum}, where we order the eigenvalues by descending real part, i.e., $\lambda_1(X)$ is the eigenvalue with the largest real part, counting multiplicity. In case that the real part of two eigenvalues is equal, we subsort them by ascending imaginary part. In case that the real part of some eigenvalues is equal, we subsort them by ascending modulus. Further, $X$ is said to be \emph{positive semidefinite}, $X \succeq 0$, if $X = X^\transp$ and $\sigma(X) \subset \mathbb{R}_{\geq 0}$. Analogously, we define positive and positive definite matrices, $X \in \mathbb{R}_{>0}$ and $X \succ 0$, respectively. Letting $I_n$ denote the identity matrix in $\mathbb{R}^{n \times n}$, the \emph{Kronecker sum} of two matrices $X \in\mathbb{R}^{n \times n}$ and $Y \in \mathbb{R}^{m \times m}$ is given by $X \oplus Y := (X \otimes I_m) + (I_n \otimes Y)$, where $\otimes$ stands for the Kronecker product. 
	
	For a real valued function $g: \mathbb{R} \to \mathbb{R} \cup \{ -\infty \}$, we say that it is \emph{concave} if $g(\lambda x + (1-\lambda)y) \geq \lambda g(x) + (1-\lambda) g(y)$ for all $0 \leq \lambda \leq 1$. The set of all concave functions will be denoted by $\mathcal{S}_{\text{c}}$, the set of all \emph{nonnegative functions} is given by $\mathcal{S}_{\geq 0} := \{g: \mathbb{R} \to \mathbb{R}_{\geq0}\}$. Further, the set of all integrable functions will be denoted by $L_1$. Then, the \emph{convolution} of two real-valued functions $g$ and $u$ is defined as $(g \ast u)(t) = \int_{-\infty}^\infty f(t-\tau) g(\tau) d \tau$ and for $\mathcal{S} \subset \mathbb{R}$, the (1-0) indicator function is defined as 
	\begin{align*}
	\mathds{1}_{\mathcal{S}}(x) = \begin{cases}
	1 & x \in \mathcal{S}\\
	0 & x \notin \mathcal{S}
	\end{cases}
	\end{align*}
	Finally, we define $\dot{g}:= \dfrac{d}{dt}g $, $\ddot{g} = \dfrac{d^2}{dt^2}g$, $\dddot{g} = \dfrac{d^3}{dt^3}g$ to be the first, second and third derivative of a real valued function $g$.

	\subsection{State-space realizations}
	A LTI state-space model
	\begin{equation}\label{eq:SISO}
	\begin{aligned}
	\dot{x} &= Ax + bu\\
	y &= cx
	\end{aligned} 
	\end{equation}
	where $A \in \mathbb{R}^{n\times n}$, $b,c^\transp \in \mathbb{R}^n$ defines  a unique causal LTI system with impulse response   $g(t) = ce^{At}b \mathds{1}_{[0,\infty)}$. The triple $(A,b,c)$ is called a \emph{realization} of this impulse response. Further, we also refer to $(A,b,c)$ as an LTI system and mean \cref{eq:SISO}. The following proposition is crucial for the derivation of our results. 
	\begin{prop}[Impulse response product~\cite{ebhihara2018analysis}]
		\label{thm:impprod}
		For $A_1 \in \mathbb{R}^{n_1 \times n_1}$, $b_1,c_1^\transp \in \mathbb{R}^{n_1}$, $A_2 \in \mathbb{R}^{n_2 \times n_2}$ and $b_2,c_2^\transp \in \mathbb{R}^{n_2}$, let $g_1(t) = c_1e^{A_1t}b_1 \mathds{1}_{[0,\infty)}$ and $g_2(t) = c_2e^{A_2t}b_2 \mathds{1}_{[0,\infty)}$. Then $g_1(t)g_2(t) \mathds{1}_{[0,\infty)} = \bar{c} e^{\bar{A}t}\bar{b} \mathds{1}_{[0,\infty)}$, where
		\begin{align*}
		\bar{A} = A_1 \oplus A_2, \quad \bar{b} = b_1 \otimes b_2, \quad \bar{c} = c_1 \otimes c_2.
		\end{align*}
	\end{prop} 
		
	For the ease of exposition, we only consider {\it causal} LTI systems in the remainder of this  paper, that is, 	we assume $g(t) = \mathds{1}_{[0,\infty)}g(t)$.

	\section{Externally positive LTI systems}
	\label{sec:ex_pos}
	Next we review the concept of external positivity.
	\begin{defn}[External positivity]
	An LTI system with impulse response $g \mathds{1}_{[0,\infty)}$ is called \emph{externally positive} if
	\begin{equation*}
	\forall u \in \mathcal{S}_{\geq 0}: \ g  \ast u \in \mathcal{S}_{\geq 0}
	\end{equation*}
	\end{defn}
%\begin{rem}
	The set of all externally positive system defines a convex cone. It is  closed under parallel as well as serial interconnection. 
%\end{rem}
We review two classical equivalent definitions of external positivity.
	\begin{lem}[\cite{farina1995necessary,ohta1984reachability}]
		\label{lem:nonneg_imp}
	An LTI system  is externally positive if and only if 
	its impulse response is nonnegative. 
	\end{lem}
\begin{lem}
	\label{lem:ex_mono}
An LTI system with impulse response $g $ is externally positive if and only if for all monotonically increasing $u \in \mathcal{S}_{\geq 0}$ it holds that $y= g   \ast u \in \mathcal{S}_{\geq 0}$ is monotonically increasing.
\end{lem}
For completeness, a proof of \cref{lem:ex_mono} is given in \Cref{proof:ex_mono}. 
We also recall the following important consequence of positivity.
\begin{prop}
	\label{prop:dominantpole}
	 If  $(A,b,c)$ is an externally positive LTI system, then $\lambda_1(A) \in \mathbb{R}$.
\end{prop}
This proposition links external positivity to the internal positivity property of mapping a cone to a cone in the state-space.  Perron-Frobenius theory shows that matrices that contract a cone also have a dominant eigenvalue and an  eigenvector  in the interior of the cone \cite{luenberger1979introduction}.

Verifying external positivity is hard. Nevertheless, there exist several sufficient tests \cite{anderson1996nonnegative,grussler2014modified,altafini2016minimal,grussler2012symmetry}. 
The following test from \cite{grussler2012symmetry} is particularly tractable.
\begin{prop}[Sufficient test for external positivity]
	\label{prop:ex_pos_test}
	Let $(A,b,c)$ be an LTI system and assume that there exists $Q = Q^\transp \in \Rnn$ and $\gamma \in \mathbb{R}$ such that
	\begin{subequations}
		\begin{align}
		&{A}^\transp Q + Q {A} + 2\gamma Q \preceq 0\\
		&{b}^\transp Q {b} \leq 0\\
		&Q + {c}^\transp {c} \succ 0\\
		&{c}{b} \geq 0\\
		&\lambda_{n-1}(Q) > 0 > \lambda_{n}(Q)
		\end{align}	
	\end{subequations}
	Then $(A,b,c)$ is externally positive, i.e., $\forall t \geq 0 : ce^{At}b \geq 0$.
\end{prop}

\section{Strongly unimodal LTI systems}
\label{sec:strong_uni}
In the following, we introduce the class of strongly unimodal LTI systems.

	\begin{defn}[Unimodality]
	A function $g: \mathbb{R} \to \mathbb{R}$ is called  \emph{unimodal} if one of the following equivalent conditions hold:
	\begin{enumerate}
		\item $g$ has a unique local maximum, i.e. there exists a mode $m \in \mathbb{R}$ such that $f$ is montonotonically increasing on $(-\infty, m]$ and montonically decreasing on $[m,+\infty)$.
		\item $g$ is quasi-concave, i.e., $$g(\lambda x + (1-\lambda) y) \geq \min \lbrace g(x), g(y) \rbrace$$ for all $x,y$ and $\lambda \in [0,1]$.
		
		The set of all unimodal functions is denoted by $\mathcal{S}_{\text{qc}}$.
	\end{enumerate}
\end{defn}
\begin{defn}[Strong unimodality]
	\label{def:strong_unimod}
	An LTI system with impulse response $g  $ is called \emph{strongly unimodal} if 	
	\begin{equation*}
	\forall u \in \mathcal{S}_{\text{qc}}: g   \ast u \in \mathcal{S}_{\text{qc}}
	\end{equation*}
\end{defn}
	 The impulse response of a strongly unimodal LTI system is certainly unimodal (approximate the Dirac impulse with the unimodal Dirac sequence, $\delta_\epsilon(t) = \frac{1}{2 \pi \epsilon} e^{-\frac{t^2}{\epsilon}}$ for $\epsilon > 0$ and apply the definition). However, unimodality of the impulse response is not sufficient.   This observation was first made  by Ibragimov \cite{ibragimov1956composition}, who introduced the terminology of strong unimodality in the context of probability distributions.

\begin{defn}[Log-concavity]
	$g \in \mathcal{S}_{\geq 0}$ is called \emph{log-concave} if for all $x,y \in \mathbb{R}$ and $\lambda  \in [0,1] $: $$g(\lambda x + (1-\lambda) y) \geq g(x)^\lambda g(y)^{1-\lambda}.$$
	Equivalently, $g$  is log-concave if and only if
	$g(x) = e^{\phi(x)}$ for some $\phi \in \mathcal{S}_c$, i.e., $\log(g) \in \mathcal{S}_{\text{c}}$. The set of all log-concave functions is denoted by $\mathcal{S}_{\text{logc}}$
\end{defn}

\begin{prop}[Log-concavity and unimodality \cite{ibragimov1956composition,dharmadhikari1988unimodality}]
	\label{prop:log_conv_unimod}
	$g \in L_1 \cap \mathcal{S}_{\text{logc}}$ if and only if 
	\begin{equation}
	\forall u \in \mathcal{S}_{\text{qc}}: g \ast u \in \mathcal{S}_{\text{qc}}.
	\end{equation}
\end{prop}

Thus, an LTI system is strongly unimodal if and only if its impulse response is log-concave. This means that strongly unimodal systems are a subset of externally positive systems \cref{lem:nonneg_imp}. In particular, it holds that
\begin{equation}
\mathcal{S}_{\text{c}} \cap \mathcal{S}_{\geq 0} \subset \mathcal{S}_{\text{logc}} \subset \mathcal{S}_{\text{qc}} \label{prop:log_conc_unimod}.
\end{equation}

We note that many unimodal density functions are also log-concave, e.g., for the exponential distribution, normal distribution, Laplace distribution, etc. \cite{boyd2004convex,karlin1968total,dharmadhikari1988unimodality}. In probability theory, log-concave density functions form a set of {\it well-behaved}
unimodal density functions \cite{Samworth2017}.

The next results reformulate log-concavity as a positivity condition.
\begin{lem}
	\label{lem:twice_diff}
	Let $g \in \mathcal{S}_{\geq 0}$ be twice-differentiable and $\mathcal{I} \subset \mathbb{R}$ be an interval. Then $g  \in \mathcal{S}_{\text{logc}}$ if and only if
	$$\forall t \in \mathcal{I}: \ \dot{g}(t)^2 -  g(t) \ddot{g}(t) \geq 0.$$
\end{lem}
\begin{proof}
	Follows by \cite[Sec.~3.5.2]{boyd2004convex} and the fact that if $g \in  \mathcal{S}_{\text{logc}}$ is then $g \mathds{1}_{\mathcal{I}} \in \mathcal{S}_{\text{logc}}$.
\end{proof}

\begin{prop}
	\label{prop:strong_uni_sys}
	A causal LTI system with impulse response $g \in L_1$ is strongly unimodal if and only if $g \mathds{1}_{[0, \infty)} \in \mathcal{S}_{\geq 0}$ and 
	$$\forall t \geq 0: \dot{g}(t)^2 -  g(t) \ddot{g}(t) \geq 0.$$
\end{prop}

With this proposition, one immediately verifies that any externally positive first-order system is also strongly unimodal. 
One also obtains the following test for second-order systems.
\begin{cor}
	\label{cor:second_order}
	Let $g$ be the impulse response of a causal stable LTI second-order system. Then the system  is strongly unimodal if and only if $g \in \mathcal{S}_{\geq 0}$ and $$ \dot{g}(0)^2 -  g(0) \ddot{g}(0) \geq 0.$$
\end{cor}
A proof to \cref{cor:second_order} is provided in \Cref{proof:second_order}.

\begin{example}[Mass-spring-damper system]
	\label{ex:MSD}
	Strong unimodality prevents oscillations in the step response of a system. As an illustration,  the classical \emph{mass-spring-damper system} with external force $u$ (see~\Cref{fig:MSD}), is modelled by the differential equation
	\begin{align}
	\ddot{x} + \frac{\beta}{m} \dot{x} +  \frac{k}{m} x = u \label{eq:MSD}
	\end{align}
	where $x$ stands for the displacement of the mass and $m,k,\beta >0$ denote the mass, spring and damping coefficients, respectively. Letting, $p := \sqrt{\beta^2 - 4k}$, the (causal) impulse response $g $ of this system is 
	\begin{equation}
	\label{eq:MSD_imp}
  g(t) = \frac{m}{p }\left(e^{-\frac{(\beta - p)t}{2}}-e^{- \frac{(\beta+p)t}{2}}\right). 
	\end{equation}
	In the overdamped case, $p \geq 0$, it follows that $g \in \mathcal{S}_{\geq 0}$. Further, since $\dot{g}(0)^2 -  g(0) \ddot{g}(0) = 1$, \cref{cor:second_order} implies that the system is also strongly unimodal. An example output for a unimodal input can be found in \cref{fig:imp_MSDs}. Thus, strong unimodality requires the mass-spring-damper system to be overdamped. This will be made even clearer in \cref{thm:poles}. 
\end{example}

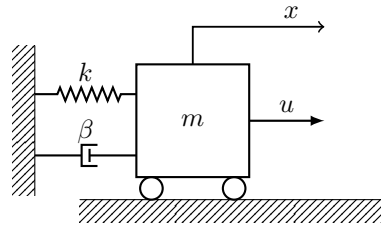
\begin{figure}
	\hspace*{.02 cm}
\begin{center}
	\begin{tikzpicture}[every node/.style={draw,outer sep=0pt,thick},force/.style={>=latex,draw=black,fill=black,thick}]
	\tikzstyle{spring}=[thick,decorate,decoration={zigzag,pre length=0.3cm,post length=0.2cm,segment length=4}]
	\tikzstyle{damper}=[thick,decoration={markings,  
		mark connection node=dmp,
		mark=at position 0.5 with 
		{
			\node (dmp) [thick,inner sep=0pt,transform shape,rotate=-90,minimum width=8pt,minimum height=3pt,draw=none] {};
			% Open part
			\draw [thick] ($(dmp.north east)+(3pt,0)$) -- (dmp.south east) -- (dmp.south west) -- ($(dmp.north west)+(3pt,0)$);
			%T part of Damper
			\draw [thick] ($(dmp.north)+(0,-2.5pt)$) -- ($(dmp.north)+(0,2.5pt)$);
		}
	}, decorate]
	\tikzstyle{ground}=[fill,pattern=north east lines,draw=none,minimum width=0.75cm,minimum height=0.3cm]
	
	% Wall
	\node (wall) [ground, rotate=-90, minimum width=2cm,yshift=-2cm] {};
	\draw (wall.north east) -- (wall.north west);
	
	% Rectengular
	\node (M) [minimum width=1.5cm, minimum height=1.5cm,xshift = .25 cm] {$m$}; 
	
    % x coordinate
 
%	\node [above = of M, draw = none, xshift = 1.5 cm, yshift = -.5 cm]{$x$};
	\draw[->, line width = .7] (M) ++ (0,.75 cm) -- ++(0,.5 cm)  -- ++(1.75,0) node[near end, above, draw = none]{$x$};
	
	% Ground
	\node (ground) [ground,anchor=north,yshift=-0.3cm,minimum width=4 cm,xshift = .5 cm] at (M.south) {};
	\draw (ground.north east) -- (ground.north west);
	
	% Roles
	\draw [thick] (M.south west) ++ (0.2cm,-0.15cm) circle (0.15cm)  (M.south east) ++ (-0.2cm,-0.15cm) circle (0.15cm);

	\draw [spring] (wall.157) -- ($(M.north west)!(wall.157)!(M.south west)$) node[above, draw = none, midway,yshift = 1.5 pt]{$k$};
	\draw [damper] (wall.18) -- ($(M.north west)!(wall.18)!(M.south west)$) node[above, draw = none, midway, yshift = 1.5 pt]{$\beta$};
	
	\draw[force,->] (M.west) ++ (1.5 cm,0) -- ++(1,0) node[above,draw = none, midway] {$u$};

	\end{tikzpicture}
\end{center}
	\caption{Mass-spring-damper system with spring and damping coefficients $k, \beta \geq 0$, mass $m > 0$ and external force $u$ is strongly unimodal if $\beta^2 \geq 4k$.}
	\label{fig:MSD}
\end{figure}

\section{State-space characterization of strong unimodality}
\label{sec:state_space}
In this section, we present our main results on strongly unimodal systems. Our first result shows that strong unimodality of a state-space model is equivalent to external positivity of a compound state-space model.

	\begin{thm}[State-space characterization]
		\label{thm:log_sys}
		Let $(A,b,c)$ be the realization of a causal impulse response $g (t) = ce^{At}b \mathds{1}_{[0,\infty)}(t)$. Then $\dot{g}(t)^2-g(t) \ddot{g}(t) $ is the impulse response
		of the state-space model $(\bar{A},\bar{b},\bar{c})$, where
		\begin{align}
		\bar{A} = A \oplus A , \quad \bar{b} =
		Ab \otimes Ab - b \otimes A^2b, \quad \bar{c} =
		c \otimes c, \label{eq:log_state}
		\end{align}
		i.e., $(A,b,c)$ is strongly unimodal if and only if $(A,b,c)$ and $(\bar{A},\bar{b},\bar{c})$ are externally positive systems. 
		
		A minimal realization $(\tilde{A},\tilde{b},\tilde{c})$ of $(\bar{A},\bar{b},\bar{c})$ has the following poles:
		\begin{equation}
		\sigma(\tilde{A}) \subset \{\lambda_i(A) + \lambda_j(A): j > i\}. \label{eq:set_poles}
		\end{equation}
		Further, if $(A,b,c)$ is minimal, then equality holds in \cref{eq:set_poles}.
	\end{thm} 
\cref{thm:log_sys} is proven in \Cref{proof:log_sys}. Note that a tractable, sufficient test for external positivity of $(A,b,c)$ and $(\tilde{A},\tilde{b},\tilde{c})$ is given in \cref{prop:ex_pos_test}. 

Next we present a key property of  strongly unimodal LTI systems.
\begin{thm}[Dominant poles]
	\label{thm:poles} 
If $(A,b,c)$ is the minimal realization of a strongly unimodal LTI system, then it has two dominant real poles, that is, $\lambda_1(A)  \in \mathbb{R}$ and $ \lambda_2(A)
 \in \mathbb{R}$. \end{thm}

This property illustrates how much strong unimodality restricts positivity: externally positive sytems require one dominant real pole, whereas stronly unimodal systems have two dominant real poles. The property also provides a  mathematical justification for our damping interpretation in \Cref{ex:MSD} for unimodal LTI systems of arbitrary order. In particular, if the system is of order three, then the three poles are necessarily real.

\begin{example}
	Consider the mass-spring-damper system \cref{eq:MSD} in series with an integrator. The dynamics are described by
	\begin{align}
	\dddot{x} + \frac{\beta}{m} \ddot{x} +  \frac{k}{m} \dot{x} = u \label{eq:MSD_int}
	\end{align}
	and the impulse response $g$ is
	\begin{equation}
	\label{eq:MSD_int_imp}
	g(t) = \frac{m}{p} \int_{0}^{t} \left(e^{-\frac{(\beta-p)\tau}{2}}-e^{- \frac{(\beta+p)\tau}{2}}\right) d \tau.
	\end{equation}
	In the underdamped case, $p < 0$, the integrand undergoes a harmonic damped oscillation with an initial positive displacement (spring extension), which is why the system is externally positive  However, due to the negative displacement phases of the integrand (spring contraction), $g$ inherits those oscillations, and is therefore not unimodal. As we noticed earlier, unimodality of the impulse response is necessary for a system to be strongly unimodal. This fact is also visualized by the example input \cref{subfig:MSD_input} with corresponding output \cref{subfig:output_MSD}.
\end{example}

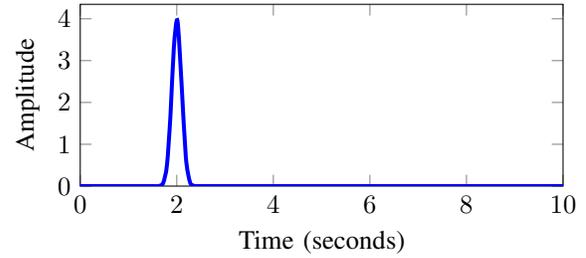
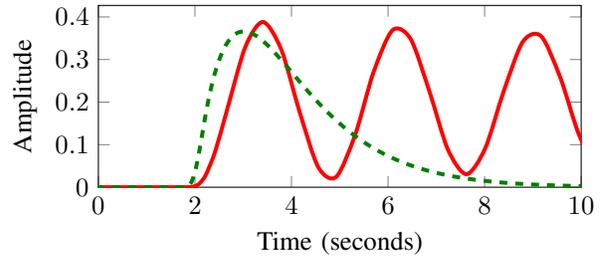
\begin{figure}[t]
	\hspace*{0.02 cm}
\begin{center}
	\begin{tikzpicture}
	\begin{groupplot}[group style={group name=my plots, group size=1 by 2,vertical sep=2.2 cm}, height = 4.0 cm, width = 8 cm]
	\nextgroupplot[
	xlabel= {Time (seconds)},
	ylabel = {Amplitude},
	ymin= 0,
%	ymax=0.1,
%	%ymax=0.08,
%	grid = both,
	xmin = 0,
	xmax = 10,
	scaled ticks=false, 
%	ymode = log,
	tick label style={/pgf/number format/fixed}]
	\addplot[smooth, color = FigColor1] file{u_msd_gauss.txt};
	\coordinate (nl) at (current axis.north);
	\nextgroupplot[
	xlabel= {Time (seconds)},
ylabel = {Amplitude},
	ymin=0,
	xmin = 0,
	xmax = 10,
	scaled ticks=false, 
	tick label style={/pgf/number format/fixed}]
	\addplot[smooth, color = FigColor2] file{y_msd_gauss.txt};
		\label{output:MSD_gauss_int}
			\addplot[smooth, color = FigColor3, dashed] file{y_msd_gauss_over.txt};
		\label{output:MSD_gaus_over}
	\coordinate (nl) at (current axis.north);
	\end{groupplot}

	\node[text width=6cm ,align=left,anchor=north] at ([yshift= -8 mm]my plots c1r1.south) {\subcaption{Unimodal input to \cref{eq:MSD,eq:MSD_int} \label{subfig:MSD_input}}};%
	\node[text width=8cm ,align=left,anchor=north] at ([yshift=-8 mm]my plots c1r2.south) 
	{\subcaption{Outputs \ref{output:MSD_gaus_over} to \cref{eq:MSD} and \ref{output:MSD_gauss_int} to \cref{eq:MSD_int}.\label{subfig:output_MSD}}};%
	\end{tikzpicture}
	\caption{(a) Input $u(t) = 
	\frac{e^{-50(t-2)^2}}{\sqrt{2 \pi 0.01}} \mathds{1}_{[0,\infty)}(t)$ to the overdamped MSD systems \cref{eq:MSD} with $m=1=k$ and $\beta = 2$, as well as, to the integrated underdamped MSD system \cref{eq:MSD_int} with $m=1$, $\beta = 0.05$ and $k = 5$; (b) Corresponding outputs to the systems. Since the output to \cref{eq:MSD_int} is not unimodal, the system cannot be strongly unimodal. \label{fig:imp_MSDs}}
\end{center}
\end{figure}
%{\color{red} I would suggest to change this figure, with one simulation of the impulse response in a slightly overdamped situation (unimodal) and one simulation in a slightly underdamped situation (not unimodal). I do not think that the impulse response of (5) adds anything to the figure.}

\section{Linear and Non-linear interconnections}
\label{sec:inter}
What makes positivity and unimodality properties attractive is that they are not restricted to linear models. Here we illustrate some interconnection properties that involve LTI models and static nonlinearities. We first recall the following two results.
\begin{lem}[Closedness \cite{boyd2004convex}]
	\label{lem:prop_log_conc}
	Log-concave functions are closed under convolution and multiplication.
\end{lem}
In particular, products of log-concave impulse responses lead to strongly unimodal LTI systems.

\begin{lem}[Composition \cite{boyd2004convex}]
	\label{lem:com_log_unimod}
	If $g \in \mathcal{S}_{\text{qc}}$  and $f:\mathbb{R} \to \mathbb{R}$ is monotonically increasing, then the composition $f \circ g\in \mathcal{S}_{\text{qc}}$.
\end{lem}

By adopting the definition of strongly unimodal LTI systems in \cref{def:strong_unimod}, \Cref{lem:prop_log_conc,lem:com_log_unimod} provide us with the following interconnection properties.\begin{prop}[Interconnections]
	Serial interconnections of strongly unimodal LTI systems and static monotonically increasing non-linearities are strongly unimodal. \end{prop}

\tikzstyle{neu}=[draw, very thick, align = center,circle]
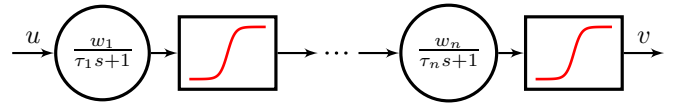
\begin{figure}
	\centering
		\begin{tikzpicture}[>=latex',circle dotted/.style={dash pattern=on .05mm off 1.2mm,
			line cap=round}]
	% System 1
	\node (a) [neu] {$\frac{w_1}{\tau_1 s+1}$}; 
	%Input 1
	\node (b) [left of=a,node distance=1.2cm, coordinate]{};
	\path[->,thick] (b) edge node[above]{$u$} node(u_inc)[above]{} (a);
	% Output 1
	\node [coordinate] (end) [right of=a, node distance=1cm]{};
	\path[->,thick] (a) edge node[above]{} (end);
	
	% Sigmoid 1
	\node (sig) [int, at = (end), minimum height = .5 cm, anchor = {west}] {
	\begin{tikzpicture}
	\begin{axis}[ticks = none,width = 2.8 cm, axis lines = none]
	\addplot[line width = 1 pt, color = FigColor2] file{sigmoid.txt};
	\end{axis}
	\end{tikzpicture}};
	
	% Output Sigmoid 1
	\node [coordinate] (mid1) [right of=sig, node distance = 1.2 cm]{};
	\path[->,thick] (sig) edge node[above]{} (mid1);
	
	%% Dotted line 
	
	% Start of dotted
	\node [coordinate] (mid2) [right of=mid1, node distance = .12 cm]{};
	
	% End of dotted
	\node [coordinate] (mid3) [right of=mid2, node distance = .3 cm]{};
	
	\path[circle dotted,very thick] (mid2) edge node[above]{} (mid3);
	
	% Input 2 start coordinate
	\node [coordinate] (b1) [right of=mid3, node distance = .12 cm]{};

	% Center of System 2
	\node [coordinate] (sys2) [right of=b1, node distance = 1.2 cm]{};
	
	% System 2
	\node (a1) [neu, at = (sys2)] {$\frac{w_n }{\tau_n s+1}$};  
	% Input 2
	\path[->,thick] (b1) edge node[above]{} (a1);
	
	% Output of System 2
	\node [coordinate] (end1) [right of=a1, node distance=1cm]{};
	\path[->,thick] (a1) edge node[above]{} (end1);

		% Sigmoid 2
		\node (sig1) [int, at = (end1), minimum height = .5 cm, anchor = {west}] {
			\begin{tikzpicture}
			\begin{axis}[ticks = none,width = 2.8 cm, axis lines = none]
			\addplot[line width = 1 pt, color = FigColor2] file{sigmoid.txt};
			\end{axis}
			\end{tikzpicture}};
	Output Sigmoid 2
\node [coordinate] (end2) [right of=sig1, node distance = 1.2 cm]{};
\path[->,thick] (sig1) edge node[above]{$v$} (end2);

	\end{tikzpicture}

\caption{Firing rate model for a serial interconnection of $n$ neurons, with rate coefficients and input weights $ \tau_1, w_1, \dots,\tau_n,w_n > 0$, is a strongly unimodal non-linear system.  \label{fig:neuron}}

\end{figure}

\begin{example}[Firing rate model of a neuron]
	An example of such non-linear systems is the serial interconnection of neurons \cref{fig:neuron}, where the output firing rate $v$ of each neuron is modeled by \cite{dayan2001theoretical}
	\begin{equation*}
	\begin{aligned}
	\tau \dot{I_s} &= - I_s + w u\\
	v &= F(I_s)
	\end{aligned}
	\end{equation*}
	where $u$ is an input rate, $\tau,w > 0$ are rate and weight coefficients, $I_s$ is the synaptic current and $F$ is a static non-linear activation function, e.g., the sigmoid function $F(x) = \frac{1}{1+e^{-x}}$. By our interconnection rules, the serial interconnection of such systems is strongly unimodal.
\end{example}

A contrario, the next result shows that parallel interconnections of strongly unimodal LTI systems are not necessarily strongly unimodal. This is in contrast to positive systems. The following result shows that this is the case even for first-order models.
\begin{lem}
	\label{thm:symm_sys}
	Let  $g(t) = \sum_{i=1}^n b_i e^{-\alpha_i t}$ with $b_i, \alpha_i > 0$ and $n \geq 2$ be such that $\alpha_{i} \neq \alpha_{i+1}$ for all $i$. Then,
	\begin{align*}
	\forall t \geq 0: \ \dot{g}(t)^2-g(t) \ddot{g}(t) < 0.
	\end{align*}	
	Therefore, $g \mathds{1}_{[0,\infty)} \in \mathcal{S}_{\geq 0} \setminus \mathcal{S}_{\text{logc}}$, which implies that strongly unimodal LTI systems are not closed under parallel interconnection. %Hence, a linear system \cref{eq:SISO} with minimal realization $A = A^T$ and $b = c^\transp$ and $n \geq 2$ is externally positive, but not log-concave.
\end{lem}

The following result shows that the difference of two positive systems can be strongly unimodal.
\begin{lem}
	\label{prop:diff}
	Let $g = b_1 e^{-\alpha_1 t} - \sum_{i=2}^n b_i e^{-\alpha_i t}$ with $b_i, \alpha_i > 0$ and $n \geq 1$ be such that
	\begin{equation}
	\forall j \geq 2: \ 2 (\alpha_{1} - \alpha_j)^2 \geq \max_i (\alpha_i -\alpha_j)^2 \label{eq:ass}
	\end{equation}
	Then, $g \mathds{1}_{[0,\infty)} \in \mathcal{S}_{\geq 0}$  if and only if $g \mathds{1}_{[0,\infty)} \in \mathcal{S}_{\text{logc}}$. 
\end{lem}
\begin{rem}
	By \cref{thm:symm_sys,prop:diff}, we can see that while the sum of two first order strongly unimodal LTI systems is not strongly unimodal, the difference preserve strong unimodality if it is externally positive. 
\end{rem}

Finally note that \cref{lem:prop_log_conc} also gives the following stronger result for log-concave inputs.
\begin{prop}
	Let $g $ be the impulse response to a strongly unimodal LTI system. Then,
	$$\forall u \in \mathcal{S}_{\text{logc}}:  g  \ast u \in \mathcal{S}_{\text{logc}}.$$
\end{prop}

\section{Concluding remarks}
\label{sec:conc}
This paper has introduced the class of strongly unimodal systems, which is characterized by preserving unimodality from input to output. Our main result is that unimodality of a state-space model is equivalent to positivity of a compound state-space model (\cref{thm:log_sys}). We have also shown that unimodal systems are a subclass of externally positive systems. As a main property, they have {\it two} dominant real poles rather than {\it one}.  

In future work, we would like to generalize our results with the theory of total positivity \cite{karlin1968total,gantmacher1950oszillationsmatrizen}.  We anticipate that systems with a fixed number of dominant real poles can be studied via the external positivity of a compound system. This paves the way for novel system analysis tools to characterize input-output properties via the important {\it variation diminishing} concept: inputs with a certain number of variations are mapped to outputs with the same (or less) number of variations. The theory of total positivity suggests that this analysis framework is general, with plausible extensions to discrete LTI systems,  linear time-space-invariant (LTSI)  models, linear time-varying linear systems, and nonlinear systems.

\bibliography{refkpos,refopt,refpos,science}

\appendix
\label{sec:app}
\subsection{Proofs}
\subsubsection{Proof to \cref{lem:ex_mono}}
\label{proof:ex_mono}
\begin{proof} We want to show that $\dot{y}: \mathbb{R}_{\geq 0} \to \mathbb{R}_{\geq 0}$. By \cite{flanders1973differentiation}, $\frac{d}{dt} y(t) = g(0)u(t) + \int_0^t ce^{A(t-\tau)}b \dot{u}(\tau) d\tau$, where the monotonicity of $u$ implies that $\dot{u}(s)$ exist almost everywhere on $[0,t]$ \cite[Theorem~6.3.3]{cohn2013measure}. Hence, since $\dot{u}, u \in \mathcal{S}_{\geq 0}$, applying \cref{lem:nonneg_imp} proves our claim. 
\end{proof}
\subsubsection{Proof to \cref{cor:second_order}}
\label{proof:second_order}
\begin{proof}
	By \cref{lem:nonneg_imp}, it follows that $g(t) = \beta_1 e^{-\lambda_1 t} + \beta_2 e^{-\lambda_2}$, $\lambda_{1}, \lambda_2 > 0$ and $\beta_1,\beta_2 \in \mathbb{R}$. Then, 
	\begin{equation*}
	\dot{g}(t)^2 -  g(t) \ddot{g}(t) = -\beta_1 \beta_2e^{-t(\lambda_1 + \alpha_2)}(\lambda_1-\lambda_2)^2.
	\end{equation*}
\end{proof}

\subsubsection{Proof to \cref{thm:log_sys}}
\label{proof:log_sys}
\begin{proof}
	The first part is an application of \Cref{thm:impprod} and the fact that $\frac{d^k}{dt^k} g(t) = cA^k e^{At}b$. In order to prove the second part, let $T \in \mathbb{C}^{n \times n}$ be such that $\hat{J} = T^{-1}AT$ is the complex Jordan form of $A$. Then, with $\hat{b} := T^{-1}b$ and $\hat{c} := cT$,

	\begin{multline*}
\left[\frac{d}{dt} g(t)\right]^2-g(t) \frac{d^2}{dt^2} g(t)   \\
	= {c}e^{{A}t}\left({A}{b}{b}^\transp - {b}{b}^\transp {A}^\transp \right)e^{{A}^\transp t}{A}^\transp {c}^\transp
	= \hat{c}K(t)\hat{J}^\transp \hat{c}^\transp.
%	&={c}T^{-1}e^{{J}t} T \left({A}^\transp {b}{b}^\transp - {b}{b}^\transp {A}^\transp \right) T^\transp e^{{J}^\transp t} T^{-\transp}{A}^\transp {c}^\transp.
%	&={c}T^{-1}e^{{J}t} \left({J}^\transp T^{-1}{b}{b}^\transp T^{-\transp} - T^{-1}{b}{b}^\transp T^{-\transp} {J}^\transp \right) T^\transp e^{{J}^\transp t} T^{-\transp}{A}^\transp {c}^\transp.
	\end{multline*}
	with $K(t):=e^{\hat{J}t}\left({J}\hat{b}\hat{b}^\transp - \hat{b}\hat{b}^\transp \hat{J}^\transp \right)e^{\hat{J}^\transp t}$. Since $K(t) = -K(t)^\transp$ and $\hat{J}$ in Jordan form, we conclude that $\left[\frac{d}{dt} g(t)\right]^2-g(t) \frac{d^2}{dt^2} g(t)$ only depends on exponentials of the form $e^{(\lambda_i(A)+\lambda_j(A))t}$ with $j > i$. Hence, \cref{eq:set_poles} follows by $\sigma(\tilde{A}) \subset(\bar{A})$. 
	
	To see the last claim notice that if $(\hat{J},\hat{b},\hat{c})$ is minimal, then the controllability of $(\hat{J},\hat{b})$ implies that $K_{ij}(t)\not\equiv 0$  for $i \neq j$. Thus, $\left[\frac{d}{dt} g(t)\right]^2-g(t) \frac{d^2}{dt^2} g(t)$ does not depend on $e^{(\lambda_i(A)+\lambda_j(A))t}$ for some $i > j$ if and only if for all $t\geq 0$
	\begin{align*}
	\hat{c}_i (\hat{c}\hat{J})_j K_{ij}(t) + \hat{c}_j (\hat{c}\hat{J})_i K_{ji}(t) = 0 
	\intertext{which by $K = -K^\transp$ is equivalent to} 
	\hat{c}_i (\hat{c}\hat{J})_j =  \hat{c}_j (\hat{c}\hat{J})_i .
	\end{align*}
	However, since $(\hat{J},\hat{c})$ is observable, $\hat{c}$ does not contain any zero entries and therefore in conjunction with the Jordan form of $\hat{J}$, this case cannot occur. 
\end{proof}

\subsubsection{Proof to \cref{thm:poles}}
\label{proof:poles}
\begin{proof}
	The fact that $\lambda_1(A) \in \mathbb{R}$ is inherited from the external positivity (see~\cref{lem:nonneg_imp}). Further, with the notation of \cref{thm:log_sys}, it follows that $\lambda_1(\tilde{A}) = \lambda_1(A) + \lambda_2(A)$, which by \cref{lem:nonneg_imp} has to be real. The last claim is then a trivial consequence.
\end{proof}

\subsubsection{Proof to \cref{thm:symm_sys}}
\label{proof:symm_sys}
\begin{proof}
	Obviously, $g \mathds{1}_{[0,\infty)}$ is nonnegative as the sum of nonnegative functions. Further, for all $t\geq 0$
	\begin{multline*}
		\dot{g}(t)^2-g(t) \ddot{g}(t) = \sum_{i=1}^n \sum_{j=1}^n b_i b_j e^{-(\alpha_i+\alpha_j)t} (\alpha_i \alpha_j-\alpha_j^2) \\
	= - \frac{1}{2}\sum_{i=1}^n \sum_{j=1}^n b_i b_j e^{-(\alpha_i+\alpha_j)t} (\alpha_{i} - \alpha_j)^2 < 0
	\end{multline*}

	Therefore, by \cref{lem:twice_diff}, $g$ is not log-concave, which by \cref{prop:strong_uni_sys} implies that the parallel interconnection of first order log-concave systems is not log-concave. 
\end{proof}

\subsubsection{Proof to \cref{prop:diff}}
\label{proof:diff}
\begin{proof} We only need to show the case with  $g \mathds{1}_{[0,\infty)} \in \mathcal{S}_{\geq 0}$.
	Then for $t \geq 0$
	\begin{align*}
	&\dot{g}(t)^2-g(t) \ddot{g}(t) \\
	& = \sum_{j=1}^n b_1 b_j e^{-(\alpha_1+\alpha_j)t} (\alpha_{1} - \alpha_j)^2\\ 
	& - \frac{1}{2}\sum_{i=2}^n \sum_{j=2}^n b_i b_j e^{-(\alpha_i+\alpha_j)t} (\alpha_{i} - \alpha_j)^2 \\
	& \geq  \sum_{i=2}^n b_i \sum_{j=2}^n b_j e^{-(\alpha_1+\alpha_j)t} (\alpha_{1} - \alpha_j)^2\\ 
	& - \frac{1}{2}\sum_{i=2}^n b_i  \sum_{j=2}^nb_j e^{-(\alpha_i+\alpha_j)t} (\alpha_{i} - \alpha_j)^2 \geq 0\\
	&\geq \sum_{i=2}^n b_i  \sum_{j=2}^n b_j e^{-(\alpha_i+\alpha_j)t} \left((\alpha_{1} - \alpha_j)^2 - \frac{1}{2}(\alpha_{i} - \alpha_j)^2\right) \\
	& \geq 0,
	\end{align*}	
	where the last inequality follows by \cref{eq:ass} and the other inequalities are a consequence of $g \mathds{1}_{[0,\infty)} \in \mathcal{S}_{\geq 0}$, which implies that $b_1 \geq \sum_{i=2}^n b_i$ and $\alpha_1 \leq \alpha_i$ for all $i$.
\end{proof}
\end{document}